\documentclass[11pt,reqno]{amsart}
\usepackage[tmargin=1in,bmargin=1in,rmargin=1in,lmargin=1in]{geometry}
\usepackage[mathscr]{eucal}
\usepackage[breaklinks=true]{hyperref}

\usepackage{amsmath,amsthm,amsfonts,amssymb,tocvsec2,color}

\theoremstyle{plain}
\newtheorem{theorem}{Theorem}[section]

\newtheorem{lemma}[theorem]{Lemma}
\newtheorem{prop}[theorem]{Proposition}

\newtheorem{utheorem}{\textrm{\textbf{Theorem}}}

\newcommand{\bd}[1]{d^{[#1]}}

\theoremstyle{definition}
\newtheorem{defn}[theorem]{Definition}

\newtheorem{rem}[theorem]{Remark}
\newtheorem{example}[theorem]{Example}

\numberwithin{equation}{section}

\DeclareMathOperator{\adj}{adj}

\newcommand{\altr}[1]{\mathbb{R}^{#1}_{\rm alt}}
\newcommand{\lcp}[1]{\mathrm{LCP}(#1,q)}
\newcommand{\lcpn}[1]{\mathrm{LCP}(#1)}
\newcommand{\sol}[1]{\mathrm{SOL}(#1,q)}
\newcommand{\rn}[1]{\mathbb{R}^{#1}}
\newcommand{\pn}[1]{\mathbb{R}^{#1}_{\rm +,-}}

\begin{document}
\title{Total negativity: characterizations and single-vector tests}
\author{Projesh Nath Choudhury}
\address[P.N.~Choudhury]{Discipline of Mathematics, Indian Institute of Technology Gandhinagar,
	Palaj, Gandhinagar 382355, India}
\email{\tt projeshnc@iitgn.ac.in}

\date{\today}

\begin{abstract}
    A matrix is called totally negative (totally non-positive) of order $k$, if all its minors of size at most $k$ are negative (non-positive). The objective of this article is to provide several novel characterizations of total negativity via the (a) variation diminishing property, (b) sign non-reversal property, and (c) Linear Complementarity Problem. (The last of these provides a novel connection between total negativity and optimization/game theory.) More strongly, each of these three characterizations uses a single test vector whose coordinates alternate in sign. As an application of the sign non-reversal property, we study the interval hull of two rectangular matrices. In particular, we identify two matrices $C^\pm(A,B)$ in the interval hull of matrices $A$ and $B$ that test total negativity of order $k$, simultaneously for the entire interval hull. We also show analogous characterizations for totally non-positive matrices and provide a finite set of test matrices to detect the total non-positivity property of an interval hull.
   These novel characterizations may be considered similar in spirit to
   fundamental results characterizing totally positive matrices by
   Brown--Johnstone--MacGibbon [\textit{J.\ Amer.\ Statist.\ Assoc.}, 1981] (see also Gantmacher--Krein, 1950), Choudhury--Kannan--Khare [\textit{Bull.\ London Math.\ Soc.}, 2021] and Choudhury [\textit{Bull.\ London Math.\ Soc.}, 2022]. Finally, we show that totally negative/non-positive matrices can not be detected by (single) test vectors from orthants other than the open bi-orthant that have coordinates with alternating signs, via  the variation diminishing property or the sign non-reversal property.

\end{abstract}



\maketitle

\vspace*{-11mm}
\settocdepth{section}
\tableofcontents

\section{Introduction and main results}

Given integers $m,n\geq k\geq 1$, a real matrix $A\in \mathbb{R}^{m \times n}$ is called \textit{totally negative of order $k$} if all minors of $A$ of order at most $k$ are negative, and $A$ is called \textit{totally negative} if $k=\min \{m,n\}$. Similarly, one defines \textit{totally non-positive} matrices (including of order $k$).

In prior discussions of totally negative (or totally non-positive)
matrices, their analogy to totally positive/$TP$ (or totally
non-negative/$TN$)  matrices, which are matrices with all positive (or
non-negative) minors, perpetually comes up first (and which we refer to
henceforth by the standard abbreviations $TP$ and $TN$, or $TP_k$ and
$TN_k$ to distinguish them from the matrices studied in the present
paper). Indeed, $TP$ and $TN$ are widely studied because of their
numerous applications in a variety of topics in mathematics, including,
analysis, combinatorics, cluster algebras, differential equations, Gabor
analysis, matrix theory, probability, and representation theory
\cite{BGKP21,BGKP20,Bre95,fallat-john,FP12,FZ02,GK50,GRS18,K68,pinkus,Ri03,S30,S55,Whitney}. In 1937, Gantmacher--Krein \cite{gantmacher-krein} characterized $TP$ matrices using positivity of the spectra of all square submatrices. A similar characterization for totally negative matrices was given by Fallat--van den Driessche in 2000 \cite{FVD00}. They proved that if $A$ is a totally negative matrix then all the eigenvalues of each square submatrix of $A$ are real, distinct, nonzero and exactly one eigenvalue is negative. (The converse is immediate.) Similar spectral characterizations hold for $TN$ and totally non-positive matrices, using the respective density of $TP$ and totally negative matrices in these classes and the continuity of roots. Totally negative/non-positive matrices are also studied in the context of matrix factorization and spectral perturbation, see \cite{CKRU08,CKRU9,HC15}.

Another interesting property exhibited by $TP$ and $TN$ matrices is the variation diminishing property. The phrase `variation diminishing' (`variationsvermindernd' in German) was formulated by P\'{o}lya in connection with proving the following  result of Laguerre \cite{Laguerre}, on a refinement of Descartes' rule of signs \cite{Descartes}: if $f(x)$ is a (real) polynomial and $\gamma \geq 0$, then the variations
$var(e^{\gamma x} f(x))$ in the Maclaurin coefficients of $e^{\gamma x}
f(x)$ are non-increasing in $0 \leq \gamma < \infty$, and are thereby bounded
above by $var(f(x)) < \infty$. In 1912, Fekete in correspondence with P\'{o}lya reformulated Laguerre's result in the language of $TP$ and proved this using one-sided P\'{o}lya frequency sequences and their variation diminishing property. In 1930, Schoenberg \cite{S30} proved that the variation diminishing property  holds for $TN$ matrices. The characterization of $TP$ and $TN$ in terms of variation diminution was achieved in 1950 by Gantmacher--Krein \cite{GK50} which was later refined by Brown--Johnstone--MacGibbon in 1981 \cite{BJM81}.
The variation diminishing property is also satisfied by other matrices, e.g. $-A$ for $A$ $TP/TN$. This was explored by Motzkin \cite{Mot36} in his PhD thesis in 1936, and he showed that the class of matrices satisfying the variation diminishing property are precisely the sign-regular matrices -- in other words, there exists a sequence of signs $\varepsilon_k \in \{ \pm 1 \}$ such that every $k \times k$ minor of $A$ has sign $\varepsilon_k$, for all $k \geq 1$. A systematic treatment of sign-regular matrices and kernels can be found in Karlin's comprehensive monograph \cite{K68}. 

The goal of this note is to provide similar characterization results for the class of totally negative/totally non-positive matrices, i.e. ones for which all minors are negative/non-positive. More precisely, we provide novel characterizations of such matrices using (a) sign non-reversal and (b) Linear Complementarity, to go along with Motzkin's classical characterization using (c) variation diminution. In addition, we strengthen all three tests (a), (b), and (c) to work with only a single test vector, in parallel to our recent works \cite{C21,CKK21} for $TP$ matrices.

 To state these results, we need some preliminary definitions, which are used below without further reference.

\begin{defn}
Let $m,n \geq 1$ be integers.
\begin{enumerate}
	\item Given a vector $x \in \mathbb{R}^n$, let $S^-(x)$ denote the number of sign changes in $x$ after removing all zero entries. Next, assign a value of $\pm 1$ to each zero entry of $x$, and let $S^+(x)$ denote the maximum possible number of sign changes in the resulting sequence. For $0\in \mathbb{R}^n$, we define $S^+(0):=n$ and $S^-(0):=0$.
\item A square matrix $A \in \mathbb{R}^{n \times n}$ has the \textit{sign non-reversal property} on a set of vectors $T\subseteq \mathbb{R}^n$, if for all vectors $ 0\neq x \in T$, there exists a coordinate $i \in [1,n]$ such that $x_i (Ax)_i> 0$.
\item A matrix $A \in
\mathbb{R}^{n \times n}$ has the \textit{non-strict sign non-reversal
	property} on a set of vectors $T\subseteq \mathbb{R}^n$, if for all vectors $ 0\neq x \in T$, there
exists a coordinate $i \in [1,n]$ such that $x_i \neq 0$ and
$x_i (Ax)_i \geq 0$.
\item A \textit{contiguous submatrix} is a submatrix whose rows and columns are indexed by sets of consecutive integers.
\item Let $\altr{n} \subset \mathbb{R}^n$ denote the set of real vectors whose coordinates are nonzero and have alternating signs. Also define $\pn{n}$ to comprise the vectors in $\mathbb{R}^n$ with at least one positive and one negative coordinate.
\item Given a matrix $A\in \mathbb{R}^{n \times n}$ and $i,j \in [1,n]$, let $A^{ij}$ denote the determinant of the submatrix of $A$ of size $n-1$ formed by deleting the $ith$ row and $jth$ column of $A$. Let $\adj(A)$ denote the adjugate matrix of $A$.
\item If $n=1$, then define $A^{11}:=-1$ to be the determinant of the empty matrix. (This, convention will only be used in Theorem \ref{tn-lcp_2}, Proposition \ref{tnplcp2}, and their proofs.)
\item  Let $e^i\in \mathbb{R}^n$ denote the vector whose $i$-th component is $1$, and other components are zero.
\item We say that an array $X$ is $\geq 0$ (respectively $X>0,~ X\leq 0,~ X<0$) if all coordinates of $X$ are $\geq 0$ (respectively $>0,~\leq 0,~<0$).
\item Define the vector $\bd{n} := (1, -1, \dots, (-1)^{n-1})^T \in \altr{n}$.
\item Given $z = (z_1, \dots, z_n)^T \in \{ \pm 1 \}^n$, define the diagonal matrix $D_z$ whose $i$th diagonal entry is $z_i$.
\item Given two matrices $A,B \in \mathbb{R}^{m \times n}$,  and
tuples of signs $z \in \{ \pm 1 \}^m, \tilde{z} \in \{ \pm 1 \}^n$, define the
$m \times n$ matrices $|A|$, $I_{z,\tilde{z}}(A,B)$, and $C^\pm(A,B)$ via:
\[
|A|_{ij} := |a_{ij}|, \qquad I_{z,\tilde{z}}(A,B) :=
\frac{A+B}{2} - D_z \frac{|A-B|}{2} D_{\tilde{z}}, \qquad C^\pm(A,B) :=
I_{\bd{m},\; \pm \bd{n}}(A,B).
\]

\end{enumerate}
\end{defn}
Now we state our main results -- only for totally negative matrices. The penultimate section contains the analogous theorems for totally non-positive matrices. We first look at  one of the most well known and widely used properties of $TP$ matrices: their variation diminution on the test set of all nonzero real vectors \cite{BJM81}. Recently in \cite{C21}, we characterized total positivity via the variation diminishing property by providing a finite set of test vectors -- in fact a single vector for each contiguous submatrix.
Our first main result similarly characterizes total negativity using the variation diminishing property. 

\begin{utheorem}\label{TN-vandimnew}
	Let $m,n\geq 2$ be integers. Given a real $m \times n$ matrix $A$ with $A<0$, the following statements are
	equivalent:
	\begin{enumerate}
		\item $A$ is totally negative.
		\item For all $ x \in \pn{n}$, $S^+(Ax) \leq S^-(x)$. Moreover, if equality holds and $Ax \neq 0$, then the first (last) component of $Ax$ (if zero, the unique sign given in determining $S^+(Ax)$) has the same sign as the first (last) nonzero component of $x$. 
		\item Let $k=\min\{m,n\}$ and $r \in[2,k]$. For every $r
		\times r$ contiguous submatrix $A_r$ of $A$ and for any choice of nonzero vector $\alpha:=(\alpha_1,-\alpha_2,\ldots ,(-1)^{r-1}\alpha_r)^T\in \mathbb{R}^r$ with all $\alpha_i\geq 0$ (or $\leq 0$), define the vector
		\begin{equation}\label{tnvdeq1}
		x^{A_r} := \adj(A_r) \alpha.
		\end{equation}
		Then $S^+(A_rx^{A_r}) \leq S^-(x^{A_r})$. If equality occurs here, then the first (last)
		component of $A_rx^{A_r}$ (if zero, the unique sign given in determining $S^+(A_rx^{A_r})$) has the same sign as the first (last)
		nonzero component of $x^{A_r}$. 
	\end{enumerate}
\end{utheorem}

We now turn our attention to a recently developed characterization of $TP$ matrices. In joint work \cite{CKK21}, we showed that $TP_k$ matrices are characterized by the sign non-reversal property on the alternating open bi-orthant, and recently in \cite{C21} we further improved this result by providing a single test vector for each contiguous submatrix.  Our next result provides the corresponding characterization of total negativity in terms of sign non-reversal phenomena. 
\begin{utheorem}\label{tn-sign-rev_k}
Let $m,n \geq k \geq 2$ be integers, and $A \in \mathbb{R}^{m \times n}$ such that $A<0$. The following statements
are equivalent:
\begin{enumerate}
	\item The matrix $A$ is totally negative of order $k$.
	\item Every square submatrix of $A$ of size $r \in [2,k]$ has the
	sign non-reversal property on $\pn{r}$.
	\item Every contiguous square submatrix of $A$ of size $r \in [2,k]$ has the sign non-reversal property on
	$\altr{r}$.\smallskip
	
	In fact, these conditions are further equivalent to sign non-reversal at a
	single vector:\smallskip
	
	\item For every contiguous square  submatrix $A_r$ of $A$ of size $r \in [2,k]$ and for any fixed nonzero vector $\alpha:=(\alpha_1,-\alpha_2,\ldots ,(-1)^{r-1}\alpha_r)^T\in \mathbb{R}^r$ with all $\alpha_i\geq 0$ (or $\leq 0$), define the vector
	\begin{equation}\label{tnsnreq1}
		x^{A_r} := \adj(A_r) \alpha.
	\end{equation}
	Then $A_r$ has the  sign non-reversal property
	for $x^{A_r}$.
\end{enumerate}
\end{utheorem}
\begin{rem}
	A careful look at the characterization of $TP$ in \cite{C21,CKK21} reveals the importance of the signs of the entries of $A$. Indeed, if $A>0$ in Theorem \ref{tn-sign-rev_k}, the assertions $(3),(4)$ are instead equivalent to $TP_k$.
\end{rem}
As an application of Theorem \ref{tn-sign-rev_k}, we simultaneously detect the total negativity of an entire interval hull of matrices by reducing it to a set of two test matrices. 
 Recall that the \textit{interval hull} of two matrices $A, B \in \mathbb{R}^{m \times n}$, denoted by $\mathbb{I}(A,B)$, is defined as follows:
\begin{equation}
\mathbb{I}(A,B) = \{C \in \mathbb{R}^{m \times n}: c_{ij} = t_{ij} a_{ij}
+ (1 - t_{ij}) b_{ij}, t_{ij} \in [0,1]\} \label{hulleqn}.
\end{equation}

If $A\neq B$, then $\mathbb{I}(A,B)$ is an uncountable set. We say that an interval hull is totally negative (totally non-positive) of order $k$ if all the matrices in it are totally negative (totally non-positive) of order $k$. For more details about interval hulls of matrices, we refer to \cite{GAT16, GAA21}. One of the interesting questions, related to interval hulls of matrices, is to find a minimal test set which
would determine if an entire interval hull $\mathbb{I}(A,B)$ is totally negative. In \cite{Gar82}, Garloff answered this question when
(i)~the interval hull consists of square matrices ($m=n$), and
(ii)~the order of total negativity equals the size of the matrices ($k=n$). In the next result, we drop these two constraints and answer this question completely.

\begin{utheorem}\label{tn_hull_k}
	Let $m,n \geq k \geq 1$ be integers, and $A, B \in \mathbb{R}^{m
		\times n}$. Then $\mathbb{I}(A, B)$ is totally negative of order $k$
	if and only if the matrices $C^\pm(A,B)$ are totally negative of order $k$.
\end{utheorem}

\begin{rem}\label{tnintrem}
	Note that $C^+(A,B)$ and $C^-(A,B)$ are independent of the choice of $k$.
\end{rem}

%
%
%

In the last part of our discussion of total negativity, we characterize
totally negative matrices using the Linear Complementarity Problem (LCP)
which generalizes and unifies linear and quadratic programming problems
and bimatrix games. Given a real square matrix $A \in \mathbb{R}^{n
\times n}$ and a vector $q \in \mathbb{R}^n$, the {\it Linear
Complementarity Problem}, denoted by $\lcp{A}$ asks to determine, if possible, a vector $x \in \mathbb{R}^n$ such that
\begin{equation}
x \geq 0,\quad y=Ax+q \geq 0, \quad and \quad x^Ty=0.\label{lcpdefn}
\end{equation}
 Any $x$ that satisfies the above three conditions is called a
 \textit{(complementarity) solution} of $\lcp{A}$. Let $\sol{A}$ denote
 the set of all solutions of $\lcp{A}$.
The Linear Complementarity Problem has important applications in diverse
areas in mathematics, including bimatrix games, convex quadratic
programming, economics, fluid mechanics, and variational inequalities
\cite{CD68,CPS09,Cr71,L65}. A plethora of matrix classes have been
studied in connection with the Linear Complementarity Problem, for
example, $P$-matrices \cite{I66,STW58}, $N$-matrices \cite{PR90} (more
broadly, see \cite[Chapter 3]{CPS09}), and recently, $TP$ matrices
\cite{C21}. For more details about Linear Complementarity Problems and
their applications, we refer to \cite{Co68,CPS09,I70}. Our final main
result characterizes total negativity via the solution set of the LCP.
\begin{utheorem}\label{tn-lcp_1}
	Let $m,n \geq k \geq 1$ be integers. Given $A \in \mathbb{R}^{n \times n}$, the following statements
	are equivalent:
	\begin{enumerate}
		\item The matrix $A$ is totally negative of order $k$.
		\item For every square submatrix $A_r$ of $A$ of size
		$r\in [1,k]$, $\lcp{A_r}$ has exactly two solutions for all $q \in \rn{r}$ with $q>0$.
		\item For every contiguous square submatrix $A_r$ of $A$
		of size $r\in [1,k]$, $\lcp{A_r}$ has exactly two solutions for all $q \in \rn{r}$ with $q>0$.
		\item For every contiguous square submatrix $A_r$ of $A$
		of size $r \in [2,k]$ and for all $q \in \rn{r}$ with
		$q>0$, $\sol{A_r}$ does not contain two vectors which have sign patterns $\begin{pmatrix}
		+ \\ 0 \\+\\0\\\vdots
		\end{pmatrix}$ and $\begin{pmatrix}
		0 \\ + \\0\\+\\\vdots
		\end{pmatrix}$. Moreover for $r=1$ and all $i\in [1,m],~j
		\in[1,n]$, $\lcp{(a_{ij})_{1\times 1}}$ has exactly two solutions for some scalar $q>0$. \smallskip
		
	\end{enumerate}
\end{utheorem}

We improve this result by providing a characterization of totally negative matrices in terms of the LCP at certain single vector, for each contiguous submatrix.

\begin{utheorem}\label{tn-lcp_2}
	Let $m,n \geq k \geq 1$ be integers.
	Given $A \in \mathbb{R}^{m \times n}$, the following statements
	are equivalent.
	\begin{enumerate}
		\item The matrix $A$ is totally negative of order $k$.
		\item For every $r \in [1,k]$ and contiguous $r
		\times r$ submatrix $A_r$ of $A$, define the vectors
		\begin{equation}\label{tlcp}
		x^{A_r} := (A^{11}_r,0,A^{13}_r,0,\ldots)^T, \qquad q^{A_r}: =A_r x^{A_r}.
		\end{equation}
		Then $0$ and $-x^{A_r}$ are the only solutions of
		$\lcpn{A_r,q^{A_r}}$.
		
	\end{enumerate}
\end{utheorem}

We end by explaining the organization of the paper. The next three sections contain the proofs of our main results above: the variation diminishing property; the sign non-reversal and interval hull results; and the results involving the LCP. Also note that it is natural to ask for `totally non-positive' analogues of the above main results, in the spirit of $TN$ analogues of results for $TP$ matrices. In the penultimate section of the paper, we indeed provide these -- see Theorems \ref{TNP-vandimnew}, \ref{tnp-snr}, \ref{Tnp_int_k},  and Propositions \ref{tnplcp1} and \ref{tnplcp2}, respectively. In the final section, we show that test vectors from open orthants other than the alternating bi-orthant cannot be used in the above characterizations of total negativity/non-positivity. In this sense our characterization results are `best' possible.

We end on a somewhat philosophical note. The rich applications of $TP$ matrices to multiple areas like analysis, approximation theory, combinatorics, differential equations, etc.\ stem from (and make use of) the numerous desirable properties and characterizations that totally positive/non-negative matrices and kernels possess. The main results in the present paper obtain similar characterization results for totally negative and totally non-positive matrices. We hope that these results will lead to applications of these classes of matrices in other areas.

\section{Theorem \ref{TN-vandimnew}: Characterization of total negativity using variation diminution}\label{vdsec}

We begin by proving Theorem \ref{TN-vandimnew} -- i.e., to provide a characterization of totally negative matrices using variation diminution. This requires a 1968 result of Karlin for totally negative matrices (in fact strictly sign-regular matrices) which was first proved by Fekete in 1912 for $TP$ matrices, and subsequently extended by Schoenberg to $TP_k$ in 1955.
\begin{theorem}[Karlin~\cite{K68}]\label{fec}
	Let $m,n \geq k \geq 1$ be integers. Then $A \in \mathbb{R}^{m \times n}$
	is totally negative of order $k$ if and only if every $r \times r$ contiguous submatrix of $A$ has negative determinant, for $r \in [1,k]$.
\end{theorem}
\begin{proof}[Proof of Theorem \ref{TN-vandimnew}]
	We begin by showing $(1) \implies (2)$. Let $A \in \mathbb{R}^{m \times n}$ be a totally negative matrix and $x \in \pn{n}$ with $1\leq S^-(x)=k\leq n-1$. Then $x$ can be partitioned into $k+1$ components of contiguous coordinates with like signs:
	
	\begin{equation}\label{vdpart}
	(x_1,\ldots ,x_{s_1}),~~(x_{s_1 +1},\ldots ,x_{s_2}), ~~ \ldots ,(x_{s_k +1},\ldots ,x_{n}),
	\end{equation}
	with at least one coordinate in each component nonzero and all nonzero coordinates in the $i$th component having the same sign $(-1)^{i-1}$, without loss of generality. Moreover, we set $s_0=0$ and $s_{k+1}=n$. Denote the columns of $A$ by  $a^1,\ldots, a^n \in \mathbb{R}^n$, and define
	
	\[b^i:= \sum\limits_{j=s_{i-1}+1}^{s_i}\vert x_j \vert a^j, \hbox{~for~} i\in [1,{k+1}].
	\]
	One can verify that the matrix $B:=[b^1,\ldots, b^{k+1}] \in \mathbb{R}^{m \times (k+1)}$ is totally negative and $B\bd{k+1}=Ax$.
	
	With this information in hand, we now prove $S^+(Ax)\leq S^-(x)=k$. For ease of exposition, we split the remainder of the proof into two cases.
	
	\noindent \textbf{Case 1. $m\leq {k+1}.$}
	
	If $Ax\neq 0$, then $S^+(Ax)\leq {m-1} \leq S^-(x)$. If $Ax=0$, then $m \leq k$, since $Ax=B\bd{k+1}$ and $B$ is invertible if $m=k+1$. Thus $S^+(Ax)=m \leq S^-(x)$.
	
	\noindent \textbf{Case 2. $m> {k+1}.$}
	
	Set $y:=Ax=B\bd{k+1}$. If $S^+(Ax)>k$, then there exist indices $i_1<i_2<\cdots < i_{k+2}\in [1,m]$ and a sign $\epsilon \in \{+,-\}$ such that $(-1)^{r-1}\epsilon y_{i_r}\geq 0$ for $r \in [1, {k+2}]$. Since $B$ is totally negative, at least two of the $y_{i_r}$ are nonzero. Let $I=\{i_1,\ldots, i_{k+2}\}$ and define the $(k+2)\times (k+2)$ matrix \[M:=[y_I|B_{I \times [1, {k+1}]}].\]
	Then $\det M=0$, since the first column of $M$ is a linear combination of the rest. Moreover, expanding along the first column gives
	\begin{equation}
	0=\sum_{r=1}^{k+2} (-1)^{r-1}y_{i_r} \det B_{{I\setminus \{i_r\}}\times [1, {k+1}]},
	\end{equation}
	
	\noindent a contradiction, since $B$ is totally negative, all $(-1)^{r-1}y_{i_r}$	have the same sign, and at least two $y_{i_r}$ are nonzero. Thus $S^+(Ax)\leq k=S^-(x)$.
	
	It remains to prove the second part of  the assertion $(2)$. We continue our discussion using the notation in the preceding analysis. We claim that, if $S^+(Ax)=S^-(x)=k$ with $Ax\neq 0$, and $(-1)^{r-1}\epsilon y_{i_r}\geq 0$ for $r \in [1,{k+1}]$ -- as opposed to $[1,{k+2}]$, then $\epsilon=1$. Since $B$ is totally negative and $B\bd{k+1}=Ax$, so the submatrix $B_{I\times [1, {k+1}]}$ is invertible and $B_{I\times [1,~k+1]}\bd{k+1}=y_I$, where $I=\{i_1,\ldots, i_{k+1}\}$. By Cramer's rule, the first coordinate of $\bd{k+1}$ is
	\[1=\frac{\det [y_I|B_{I\times [2,{k+1}]}]}{\det B_{I\times [1, {k+1}]}}.\]
	Multiplying both sides by $\epsilon \det B_{I\times [1, {k+1}]}$ and expanding the numerator along the first column, we have
	
	\begin{equation}
	\epsilon \det B_{I\times [1, {k+1}]}=\sum_{r=1}^{k+1} (-1)^{r-1}\epsilon y_{i_r} \det B_{I\setminus \{i_r\}\times [2,{k+1}]}.
	\end{equation}
	Since the summation on the right side is negative and $B$ is totally negative, we obtain $\epsilon=1$.
	\vspace{0.5cm}	

We next show that $(3) \implies (1)$. By Theorem \ref{fec},
it suffices to show for all $r \times r$ ($r\in [1,k]$) contiguous
submatrices $A_r$ of $A$ that $\det A_r<0$. The proof is by induction on
$r \in [1,k]$. The base case holds by the hypothesis.
	
	Let $A_r$ be an $r \times r$ contiguous submatrix of $A$ with $r \in [2,k]$ and suppose that all contiguous minors of $A$ of size at most $(r-1)$ are negative. The same holds for all proper minors of $A_r$ by Theorem \ref{fec}. Define the vector $x^{A_r}$ as in \eqref{tnvdeq1}.
	Then \begin{equation}\label{tnvdeq2}
	x^{A_r}_i=   \sum_{j=1}^r (-1)^{j-1} \alpha_j \cdot (-1)^{(i-1) +
	(j-1)} A^{ji}_r=(-1)^{i-1} \sum_{j=1}^r  \alpha_j A^{ji}_r
	\end{equation}
	and  if the $\alpha_i$'s are non-negative (or non-positive), then the expression on the right is negative (or positive) for odd $i$ and positive (or negative) for even $i$. It follows that
	\begin{equation}\label{tnvdeq3}
	x^{A_r}\in \altr{r},\qquad A_r x^{A_r}=(\det{A_r})\alpha.
	\end{equation}
	
	We first claim that $A_r$ is invertible. Suppose instead that $A_r$ is singular. Then $r=S^+(A_rx^{A_r})> S^-(x^{A_r})=r-1$, a contradiction by $(3)$, and hence $A_r$ is invertible as claimed.
	
	Next we show that $\det {A_r}<0$. By \eqref{tnvdeq3}, we have \[r-1=S^+(A_rx^{A_r})= S^-(x^{A_r}),
	\]
	since the condition on the $\alpha_i$'s implies that
	$S^+(\alpha)=r-1$.
	Also, the sign of the first (last) component of
	$\frac{1}{\det A_r}A_rx^{A_r}$ (if zero, the sign that attains $S^+(\frac{1}{\det A_r}A_rx^{A_r})$)
	is positive or negative if the $\alpha_i$'s are non-negative or
	non-positive respectively. 
	According to the two lines following
	\eqref{tnvdeq2}, we know the sign of the first (last) component of
	$x^{A_r}$, and can use hypothesis~(3) to conclude with
	\eqref{tnvdeq3} that $\det{A_r} < 0$. This completes the
	induction step and also the proof of  $(3) \implies (1)$.

	Finally, we show that $(2) \implies (3)$, by induction on
	$r \in [2,k]$. We know that $A$ is totally negative of order
	$r-1$ (by hypothesis if $r=2$, and by the induction hypothesis
	combined with $(3) \implies (1)$ if $r>2$). Suppose $A_r=A_{I
	\times J}$ is an $r \times r$ contiguous submatrix of $A$, for
	contiguous sets of indices $I\subseteq [1,m]$ and $J \subseteq
	[1,n]$ with $\lvert I\rvert = \lvert J \rvert=r$. Define
	$x^{A_r}$ as in \eqref{tnvdeq1}. As above, $A_r$ is
	totally negative of order $r-1$, so the entries in each row and
	column of $\adj(A_r)$ have alternating signs. Since $\alpha \neq
	0$, \eqref{tnvdeq1} implies that $x^{A_r} \in \altr{r} $. Now
	define $x\in \pn{n}$ to have $x^{A_r}$ in contiguous positions
	$J\subseteq [1,m]$ and $0$ elsewhere. Then 
	\begin{equation}  S^+(A_rx^{A_r})\leq S^+(Ax)\quad \hbox{and}\quad S^-(x)=S^-(x^{A_r}).\label{vdeq2}
	\end{equation} By $(2)$, we have  $S^+(A_rx^{A_r}) \leq S^-(x^{A_r})$.
	
	Next, suppose that $S^+(A_rx^{A_r}) = S^-(x^{A_r})=r-1$. Then $S^+(Ax)=S^-(x)$. 
	 Since  $S^+(A_rx^{A_r}) = S^+(Ax)$, by (2), it follows that all coordinates of $Ax$ or an $S^+$-completion of $Ax$ in positions $1,2,\ldots, i_1$ (respectively, in positions  $i_r,\ldots, m$) are nonzero, and with the same sign as that of $x^{A_r}_1 ~(\hbox{respectively}~x^{A_r}_r)$. This shows  $(2) \implies (3)$.
\end{proof}

\section{Theorems \ref{tn-sign-rev_k} and \ref{tn_hull_k}: Sign non-reversal property and interval hulls for totally negative matrices}\label{tpsec}

We next prove Theorems \ref{tn-sign-rev_k} and \ref{tn_hull_k}.  To prove
Theorem \ref{tn-sign-rev_k}, we begin with a preliminary result, which
establishes the sign non-reversal property for matrices with negative
principal minors (these are known as $N$-matrices):

\begin{theorem}\cite{PR90}\label{snrn}
	Let  $n \geq 2$ be an integer and $A \in \mathbb{R}^{n \times n}$ such that $A<0$. Then $A$ has all principal minors
	negative if and only if for all $ x \in \pn{n}$, there exists
	$i \in [1,n]$ such that $x_i (Ax)_i > 0$.
\end{theorem}

\begin{proof}[Proof of Theorem \ref{tn-sign-rev_k}]
	That $(1) \implies (2)$ follows from Theorem \ref{snrn}; and that
	$(2) \implies (3)$ is immediate. Next, suppose
	$(4)$ holds. To show $(1)$, by Theorem \ref{fec} it suffices to show that the
	determinants of  all $r \times r$ contiguous submatrices of $A$
	are negative, for $r \in [1,k]$. We prove this by induction on
	$r$, where the base case holds by the hypothesis. Fix $r \in
	[2,k]$ and assume that all $(r-1)\times (r-1)$ and smaller
	contiguous minors of $A$ are negative. Let $A_r$ be an $r \times
	r$ contiguous submatrix of $A$. By the induction hypothesis, all
	the proper contiguous minors of $A_r$ are negative. In fact, all
	the proper minors of $A_r$ are negative by Theorem \ref{fec}.
	
	 For any choice of a nonzero vector $\alpha:=(\alpha_1,-\alpha_2,\ldots ,(-1)^{r-1}\alpha_r)^T$ with all $\alpha_i\geq 0$ (or $\leq 0$), define the vector $x^{A_r}$ as in \eqref{tnsnreq1}. Then \begin{equation}\label{tnsnreq2}
		x^{A_r}_i=   \sum_{j=1}^r (-1)^{j-1} \alpha_j \cdot
		(-1)^{(i-1) + (j-1)} A^{ji}_r=(-1)^{i-1}\sum_{j=1}^r
		\alpha_j A^{ji}_r
	\end{equation}
	and  if the $\alpha_i$'s are non-negative (or non-positive), then the expression on the right is negative (or positive) for odd $i$ and positive (or negative) for even $i$. Thus $x^{A_r} \in \altr{r}$, and 
	\begin{equation}\label{tnsnreq3}
		A_rx^{A_r}=(\det A_r) \alpha.
	\end{equation}
	By the hypothesis, \eqref{tnsnreq2}, and \eqref{tnsnreq3}, we have $i_0 \in [1,r]$ such that
	\begin{equation}\label{tnsnreq4}
	0< x^{A_r}_{i_0}(A_rx^{A_r})_{i_0} =(-1)^{i_0-1} (\det A_r)\alpha_{i_0}x^{A_r}_{i_0}.
	\end{equation}
	From this, it follows that $\det A_r<0$, which completes the
	induction step. Hence $(4) \implies (1)$.

	Finally, we show that $(3) \implies (4)$, by induction on $r \in
	[2,k]$. We know $A$ is totally negative of order $r-1$ (by
	hypothesis if $r = 2$, and by the induction hypothesis and $(4)
	\implies (1)$ if $r>2$). As in \eqref{tnsnreq2} and lines following it,  $x^{A_r}\in \altr{r}$ since
	$\alpha \neq 0$. Now $(3)$ immediately implies $(4)$.
\end{proof}

To prove Theorem \ref{tn_hull_k}, we need two preliminary lemmas; the first is straightforward.
\begin{lemma}\label{int-lem1}
	Let $m, n \geq 1$ and $A,B \in \mathbb{R}^{m \times n}$. Then $I_{z,\tilde{z}}(A,B) \in \mathbb{I}(A,B)$ for all $z \in \{ \pm 1 \}^m, \tilde{z} \in \{ \pm 1 \}^n$.
\end{lemma}

\begin{lemma}\cite{RJRG}\label{rohn_exten2}
	Let $A,B \in \mathbb{R}^{n \times n}$ and $x \in \mathbb{R}^n$. Let the vector $z \in \{ \pm 1 \}^n$ be such that $z_i
	= 1$ if $x_i \geq 0$ and $z_i = -1$ if $x_i < 0$. If $C \in
	\mathbb{I}(A,B)$, then
	\[
	x_i(Cx)_i\geq x_i (I_{z,z}(A,B) x)_i, \quad \forall i \in [1,n].
	\]
\end{lemma}

With these preliminaries at hand, we now prove Theorem \ref{tn_hull_k}.

\begin{proof}[Proof of Theorem \ref{tn_hull_k}]
	The result is immediate for $k=1$, so we suppose henceforth $m,n\geq k\geq 2$. Let $\mathbb{I}(A, B)$ be totally negative of order $k$. By  Lemma \ref{int-lem1}, $C^\pm(A,B)$ are totally negative of order $k$.
	
	Conversely, suppose $C^\pm(A,B)$ are totally negative of order $k$, and let $M \in
	\mathbb{I}(A,B)$. Then $M<0$, since $C^+(A,B),C^-(A,B)<0$. Fix $r \in [2,k]$ and a vector $x \in
	\altr{r}$. Let $M'=M_{J \times K}$ be an $r \times r$ contiguous submatrix of
	$M$, where $J\subseteq[1,m]$ and $K \subset [1,n]$ are contiguous sets of indices  with $|J| = |K| = r$. By Theorem~\ref{tn-sign-rev_k}(3), it suffices to
	show  that there exists a coordinate $i \in [1,r]$ such that $x_i (M'x)_i > 0$.
	
	To proceed further, we need some notation. Let $A', B'$ be contiguous
	submatrices of $A,B$ respectively, with the rows and columns of both indexed by $J$ and $K$, respectively. Since $M' \in \mathbb{I}(A',B')$,
	Lemma~\ref{rohn_exten2} implies there exists $z' \in \{ \pm 1 \}^r \cap \altr{r}$ such
	that
	\begin{equation}\label{Erohn}
	x_i (M' x)_i \geq x_i (I_{z', z'}(A',B')x)_i,\qquad \forall i \in [1,r].
	\end{equation}
	Since $M' = M_{J \times K}$, we extend $z'$ to $z \in \altr{m} \cap \{\pm 1 \}^m$ ($\tilde{z}\in \altr{n} \cap \{\pm 1 \}^n$) by embedding in contiguous positions $J\subseteq [1,m]$ ($K \subseteq [1,n]$) and uniquely padding by $\pm 1$ elsewhere. 
	Then $I_{z',z'}(A',B')$ is a $r \times r$ contiguous submatrix of
	\[
	I_{z,\; \tilde{z}}(A,B) := \frac{A+B}{2}
	- D_{z} \frac{|A-B|}{2} D_{\tilde{z}}
	\in \{ C^+(A,B), C^-(A,B) \}.
	\]
	By hypothesis, this matrix is totally negative of order $k$. Thus $I_{z',z'}(A',B')$ is totally negative. Using Theorem~\ref{tn-sign-rev_k}(3) and~\eqref{Erohn}, for some $i\in [1,r]$, we have
	\[
	x_i (M' x)_i \geq x_i (I_{z',z'}(A',B')x)_i > 0.
	\]
	Thus $M'$ has the sign non-reversal property on $\altr{r}$. Since $M<0$, by Theorem~\ref{tn-sign-rev_k}, $M$ is totally negative of order $k$.
\end{proof}

\begin{rem}
	Theorem~\ref{tn_hull_k} can be alternately shown using Garloff's
	results for $n \times n$ matrices in \cite{Gar82}. These results
	(also see \cite[Chapter~3]{pinkus}) have been stated using the
	checkerboard ordering $\leq^*$, wherein
	\[
	A \leq ^* B \quad \Longleftrightarrow \quad D_{\bd{n}} (B - A)
	D_{\bd{n}} \geq 0_{n \times n}.
	\]
	It is now easily seen that the composition of the hull
	$\mathbb{I}(A,B)$ does not change upon using the entrywise or the
	checkerboard ordering:
	\[
	\mathbb{I}(A,B) = \{ C : I_l \leq C \leq I_u \} = \{ C : C^-(A,B) \leq^*
	C \leq^* C^+(A,B) \}.
	\]
	This implies that the test set $\{ C^+(A,B), C^-(A,B) \}$ will
	work regardless of the ordering used -- and it will work for all
	rectangular matrices, as well as in testing the property of being
	totally negative of order $k$, for any $k \geq 1$.
\end{rem}

\section{Theorems \ref{tn-lcp_1} and \ref{tn-lcp_2}: Characterization of total negativity via the LCP}
In this section, we show Theorems \ref{tn-lcp_1} and \ref{tn-lcp_2}. To proceed, we recall a 1990 result which establishes a connection between the LCP and matrices with negative principal minors.
\begin{theorem}\cite{PR90}\label{lcpn}
	A matrix $A \in \mathbb{R}^{n \times n}$ with $A<0$ has all principal minors negative if and only if 
	$\lcp{A}$ has exactly two solutions for all $q \in \mathbb{R}^{n}$ with $q>0$.
\end{theorem}

\begin{proof}[Proof of Theorem \ref{tn-lcp_1}]
	That $(1) \implies (2)$ follows from Theorem \ref{lcpn}, while $(2)
	\implies (3) \implies (4)$ is immediate. We now show $(4)
	\implies (1)$. We first claim that $A<0$. Indeed, if  $a_{ij}\geq
	0$ for some $i\in [1,m]$ and $j \in [1,n]$, then $(0)$ is the
	only solution of $\lcp{(a_{ij})_{1\times 1}}$ for any scalar $q>0$.
	
	Next, we show that the determinants of all $r \times r$ submatrices of $A$ are negative, for $r \in [2,k]$. By Theorem \ref{tn-sign-rev_k}, it suffices to show for that every contiguous square submatrix of $A$ of size $ r \in [2,k]$ has the sign non-reversal property on $\altr{r}$. Let $r \in [2,k]$ and $A_{r}$ be an $r \times r$ contiguous submatrix of $A$ which  does not satisfy this property. Then there exists $x \in \altr{r}$ such that $x_i(A_r x)_i\leq 0$ for all $i \in [1,r]$. Let $A_rx=v$. Define $x^{\pm}:=\frac{1}{2}(\vert x \vert \pm x)$ and $v^{\pm}:=\frac{1}{2}(\vert v \vert \pm v)$. Then $x^+$, $x^-$ have sign patterns $\begin{pmatrix}
	+ \\ 0 \\+\\0\\\vdots
	\end{pmatrix}$, $\begin{pmatrix}
	0 \\ + \\0\\+\\\vdots
	\end{pmatrix}$ respectively (or vice-versa) and ${(x^+)}^T v^+ ={(x^-)}^T v^- =0$, since $x_i v_i\leq 0$ for all $i \in [1,r]$. Define
	\begin{equation}
	q:=v^+-A_rx^+=v^--A_rx^-.\label{thrmaeq1}
	\end{equation}
	Then $q>0$, since $A<0$. Thus $\lcp{A_r}$ has solutions with sign patterns $\begin{pmatrix}
	+ \\ 0 \\+\\0\\\vdots
	\end{pmatrix}$ and $\begin{pmatrix}
	0 \\ + \\0\\+\\\vdots
	\end{pmatrix}$, a contradiction. Hence every contiguous square submatrix of $A$ of size $r \in [2,k]$ has the
	sign non-reversal property on $\altr{r}$ and the result follows.
\end{proof}

\begin{proof}[Proof of Theorem \ref{tn-lcp_2}]
	That $(1) \implies (2)$ follows from Theorem \ref{tn-lcp_1}. To
	show  $(2) \implies (1)$,  by Theorem \ref{fec}, it suffices to
	show, for $r\in[1,k]$, that  all contiguous minors of $A$ of size
	$r$ are negative. The proof is by induction on $r$. Let $r=1$ and
	$A_1=(a_{ij})$ for some $i\in[1,m]$ and $j \in [1,n]$. Then
	$x^{A_1}=(-1)$ by convention, and $q^{A_1}=-(a_{ij})$. If
	$a_{ij}=0$, then $\lcpn{A_1, q^{A_1}}$ has infinitely many
	solutions, while if $a_{ij}>0$, then $\lcpn{A_1,q^{A_1}}$ has only one solution, a contradiction. Thus all entries of $A$ are negative.
	
	For the induction step, suppose $A_r$ is an $r \times r$ contiguous submatrix of $A$, with $r\in [2,k]$ and all contiguous minors of $A$ of size at most $r-1$ are negative. By Theorem \ref{fec}, all proper minors of $A_r$ are negative. Define
	\begin{equation}\label{lcpsingleeq1}
	x^{A_r}:=(A^{11}_r,0,A^{13}_r,0,\ldots)^T, \qquad z^{A_r}:=(0,A^{12}_r,0,A^{14}_r,\ldots)^T,\qquad q^{A_r}:=A_rx^{A_r}.
	\end{equation}  Then $x^{A_r},~z^{A_r}\leq 0$ and $q^{A_r}>0$.
	Thus $\bf{0}$ and $-x^{A_r}$ are two solutions of
	$\lcpn{A_r,q^{A_r}}$.
	
	We next claim that $A_r$ is invertible. If not, then $A_rx^{A_r}=A_rz^{A_r}$.
	So $-z^{A_r}$ is also a solution of $\lcpn{A_r,q^{A_r}}$, a contradiction. This shows the claim.
	
	Finally we show that $\det{A_r}<0$. Indeed, suppose $\det{A_r}>0$. Then
	\begin{center}
		$y=-A_rz^{A_r}+q^{A_r}=(\det{A_r})e^1 \geq 0 \hbox{~~~ and~~~} y^Tz^{A_r}=0$.
	\end{center}
	Thus $-z^{A_r}$ is also a solution of $\lcpn{A_r, q^{A_r}}$, a contradiction by the hypothesis (2). Thus $\det{A_r}<0$. This completes the induction step  and hence the proof.
\end{proof}
\begin{rem}
	In Theorem \ref{tn-lcp_2}, in place of the vector $x^{A_r}= (A^{11}_r,0,A^{13}_r,0,\ldots)^T$, one can take ${x^i}^{A_r} := (A^{i1}_r,0,A^{i3}_r,0,\ldots)^T$ for odd $i\in [1,r]$, or a positive linear combination of these vectors. The result still holds with a similar proof, where we define $q^{A_r}$ similarly as in \eqref{tlcp}.  
\end{rem}

\section{Characterizations for totally non-positive matrices}\label{sectnp}
This section presents the counterparts of the above main results for totally non-positive matrices as promised in the introduction. We first provide a characterization of totally non-positive matrices in terms of variation diminution.  For this, we require a 1950 density result of Gantmacher--Krein for totally negative matrices (in fact strictly sign regular matrices).

\begin{theorem}\cite{GK50}\label{tando}
	Given integers $m,n \geq k \geq 1$, the set of $m \times n$ totally negative matrices of order $k$ is dense in the set of $m \times n$ totally non-positive matrices of order $k$.
\end{theorem}

Apart from Gantmacher--Krein's density theorem, the proofs require another preliminary lemma on sign changes of limits of vectors; see e.g. \cite{pinkus} for the proof.
\begin{lemma}\label{limsc}
	Given $x=(x_1,\ldots,x_n)^T\in \mathbb{R}^n\setminus \{0\}$, we
	have
	\[
	S^+(x)+S^-(\overline{x})=n-1, \qquad \text{where} \qquad
	\overline{x}:=(x_1,-x_2,x_3,\ldots, (-1)^{n-1}x_n)^T\in
	\mathbb{R}^n.
	\]
	Moreover, if $\lim\limits_{p \to \infty} x_p =x$, then
	
	\[\liminf\limits_{p \to \infty} S^-(x_p)\geq S^-(x), \qquad \limsup\limits_{p \to \infty} S^+(x_p)\leq S^+(x).\]
\end{lemma}

\begin{theorem}\label{TNP-vandimnew}
	Given a real $m \times n $ matrix $A\leq 0$ and $m,n\geq 2$, the following statements are
	equivalent:
	\begin{enumerate}
		\item $A$ is totally non-positive.
		\item For all $x \in \pn{n}$, $S^-(Ax) \leq S^-(x)$. Moreover, if equality occurs and $Ax \neq 0$, then the first (last)
		nonzero component of $Ax$ has the same sign as the first (last)
		nonzero component of $x$. 
		\item Let $k=\min\{m,n\}$. For every square submatrix $A_r$ of $A$ of size $r \in [2,k]$, define the vector
		\begin{equation}\label{vdeqtn1}
		y^{A_r}:= \adj(A_r)\alpha, \hbox{for some}~ \alpha \in \altr{r}.
		\end{equation}
		Then $S^-(A_ry^{A_r}) \leq S^-(y^{A_r})$. If equality occurs here, then the first (last)
		nonzero component of $A_ry^{A_r}$ has the same sign as the first (last)
		nonzero component of $y^{A_r}$.
	\end{enumerate}
\end{theorem}


\begin{proof}
	We first show that $(1) \implies (2)$. Since $A$ is totally non-positive, by Gantmacher--Krein's density Theorem \ref{tando}, there exists a sequence $(A_p)$ of totally negative matrices such that 
	\[\lim\limits_{p \to \infty} A_p =A.\]
	Let $x \in \pn{n}$. By Theorem \ref{TN-vandimnew} and Lemma \ref{limsc}, we have
	\[S^-(Ax)\leq \liminf\limits_{p \to \infty} S^-(A_p x)\leq \liminf\limits_{p \to \infty} S^+(A_px)\leq \liminf\limits_{p \to \infty} S^-(x)=S^-(x).\]
	Next, assume that $S^-(Ax)=S^-(x)=k$ and $Ax\neq 0$. Then for all sufficiently large $p$, by Theorem \ref{TN-vandimnew} and Lemma \ref{limsc}, 
	\[S^-(Ax)\leq S^-(A_px)\leq S^+(A_px)\leq S^-(x)=k.\]
	
	Thus $S^-(A_px)=S^+(A_px)=S^-(x)=k$ for $p$ large enough. This implies the sign patterns in $A_px$ (for large $p$) do not depend on the zero entries. Thus the nonzero sign patterns of $A_p x$ agree with those of $Ax$. Also, by Theorem \ref{TN-vandimnew}, the sign patterns of $x$ agree with the sign patterns of $A_px$. Hence  the same holds for the sign patterns of $Ax$ and $x$.

	Next, we show that $(3) \implies (1)$. The claim that $\det{A_r}\leq 0$ for all $r \times r$ submatrices $A_r$ of $A$, is shown by induction on $r \in [1,k]$. The base case $r=1$ holds by the hypothesis.
	
	For the induction step, $A_r$ is an $r \times r$ submatrix of $A$ and all proper minors of $A_r$  are non-positive. If $\det{A_r}=0$ then there is nothing to prove, so suppose $\det{A_r} \neq 0$. Let $\alpha=(\alpha_1,\alpha_2,\ldots, \alpha_r)^T \in \altr{r}$ and define $y^{A_r}$ as in \eqref{vdeqtn1}. Then no row of $\adj(A_r)$ is zero and $y^{A_r}\in \altr{r}$. Also, $y^{A_r}_i$ is negative (or positive) if $\alpha_i$ is positive (or negative) for all $i \in [1,r]$, by \eqref{tnvdeq2}. Since $A_ry^{A_r}=(\det{A_r}) \alpha$, we have 
	\begin{center}
		$S^-(A_ry^{A_r})=S^-(y^{A_r})=r-1.$
	\end{center}
	By the last part of the hypothesis (3), $\det{A_r}<0$ and so we are done by induction.
	
	Finally, to show $(2)\implies (3)$, repeat the proof of corresponding implication in Theorem \ref{TN-vandimnew}, now working with arbitrary invertible submatrices $A_r$ of $A$ of size $r \in [2,k]$ and the vector $y^{A_r}=\adj (A_r)\alpha $, where $\alpha \in \altr{r}$ is  arbitrarily chosen.
\end{proof}
Analogous to Theorems \ref{tn-sign-rev_k} and \ref{tn_hull_k}, we next derive the sign non-reversal property and interval hull tests for totally non-positive matrices.

\begin{theorem}\label{tnp-snr}
	Let $m,n \geq k \geq 2$ be integers. Given $A \in \mathbb{R}^{m \times
		n}$ with $A\leq 0$, the following statements are equivalent:
	\begin{enumerate}
		\item The matrix $A$ is totally non-positive of order $k$.
		
		\item Every square submatrix of $A$ of size $r\in [2,k]$ has the non-strict
		sign non-reversal property on $\pn{r}$.
		
		\item Every square submatrix of $A$ of size $r \in [2,k]$ has the non-strict
		sign non-reversal property on $\altr{r}$.
		
		\item For every $r \in [2,k]$ and $r \times r$ submatrix $A_r$
		of $A$, the matrix $A_r$ has the non-strict sign non-reversal property for the vector $x^{A_r}=\adj (A_r)\alpha$, for any choice of $\alpha \in \altr{r}$.
	\end{enumerate}
\end{theorem}
\begin{proof}
	First suppose $A \in \mathbb{R}^{m \times n}$ is totally non-positive of order $k$. By
	Theorem~\ref{tando}, there exists a sequence $A^{(p)}$ of totally negative matrices of order $k$  such that $A^{(p)} \to A$ entrywise. Fix an $r \times r$ submatrix $A_r$ of $A$ for $r \in [2,k]$, and let $A^{(p)}_r$ be the submatrix of $A^{(p)}$ indexed by the same rows and columns of $A_r$. Now fix a vector $x \in \pn{r}$, and let
	$I \subseteq [1,r]$ index the nonzero coordinates of $x$. Since
	$A^{(p)}_r$ is totally negative, by Theorem~\ref{tn-sign-rev_k}(2) there exists $i_p
	\in I$ such that $x_{i_p} (A^{(p)}_r x)_{i_p} > 0$. Since $I$ is finite, there exists an
	increasing subsequence of positive integers $p_s$ such that $i_{p_s}=i_0$ for some $i_0\in I$ and for all $s \geq 1$. Hence
	\[
	x_{i_0} (A_r x)_{i_0} = \lim_{s \to \infty} x_{i_{p_s}} (A^{(p_s)}_r
	x)_{i_{p_s}} \geq 0, \qquad x_{i_0} \neq 0.
	\]
	
	\noindent This completes the proof of $(1) \implies (2)$.
	
	Next, that $(2) \implies (3)$ is immediate. Now
	assume~(4) holds. To prove $(1)$, we show by induction on $r \in [1,k]$ that the determinant of every $r \times r$ submatrix
	of $A$ is non-positive. The base case $r=1$ follows from the hypothesis. For the induction step, suppose $A_r$ is a square submatrix of $A$ of size $r \in [2,k]$ and all the proper minors of $A$ are non-positive. If $\det A_r = 0$ then we are done; else $\det(A_r) \neq 0$. Let $\alpha=(\alpha_1,\ldots,\alpha_r)^T\in \altr{r}$ and define the vector $x^{A_r}:=\adj (A_r)\alpha$. Since no row of $\adj (A_r)$ is zero, by \eqref{tnsnreq2}, $x^{A_r} \in \altr{r}$ and $x^{A_r}_i<0 ~(>0)$ if $\alpha_i >0 ~(<0)$ for all $i\in[1,r]$.  Now a similar calculation as in \eqref{tnsnreq4}, implies that for some $i_0 \in [1,r]$, we have
	\begin{center}
		$0\leq x^{A_r}_{i_0}(A_rx^{A_r})_{i_0} = (\det A_r)\alpha_{i_0}x^{A_r}_{i_0}.$
	\end{center} Since $x^{A_r}_{i_0}$ and $\alpha_{i_0}$ have opposite signs, it follows that $\det A_r<0$, which completes the proof of $(4) \implies (1)$ by induction.

Finally, to show $(3) \implies (4)$, let $A_r$ be an arbitrary invertible submatrix of size $r \in [2,k]$ and let $y^{A_r}=\adj (A_r)\alpha $, where $\alpha \in \altr{r}$ is  arbitrarily chosen. Now repeat the proof of the corresponding implication in Theorem \ref{tn-sign-rev_k}.
\end{proof}

Given Theorem~\ref{tnp-snr}, which is a totally non-positive of order $k$ analogue of
Theorem~\ref{tn-sign-rev_k}, a natural question is to seek a similar
totally non-positive (of order $k$) analogue of Theorem~\ref{tn_hull_k}. In \cite{AG16}, the authors show that under certain technical
constraints, a minimal test set of two matrices suffices to check the
total non-negativity of the entire interval hull. Our next result
removes these technical constraints from~\cite{AG16}, and holds for totally non-positive of order $k$
interval hulls for arbitrary $k \geq 1$ -- at the cost of working with a
larger (but nevertheless finite) test set:
\begin{theorem}\label{Tnp_int_k}
	Let $m,n \geq k \geq 1$ be integers, and $A,B \in \mathbb{R}^{m \times n}$. Then $\mathbb{I}(A,B)$ is totally non-positive of order $k$ if and only if the
	matrices $\{ I_{z,\tilde{z}}(A,B) : z \in \{ \pm 1 \}^m, \tilde{z} \in \{ \pm 1 \}^n
	\}$ are all totally non-positive of order $k$.
\end{theorem}
\begin{rem}
	As in Remark~\ref{tnintrem}, note that this set of test vectors is independent of the choice $k$.
\end{rem}

\begin{proof}
	We assume $m,n\geq k\geq 2$, since the result is immediate for $k=1$. If $\mathbb{I}(A,B)$ is totally non-positive of order $k$, then so are $I_{z,\tilde{z}}(A,B)$ by Lemma \ref{int-lem1}.
	
	To show the converse, we repeat the proof of Theorem~\ref{tn_hull_k}, now working with
	arbitrary submatrices $M'$ of $M \in \mathbb{I}(A,B)$ of size $r \in [2,k]$ instead of contiguous submatrices. We once again have ~\eqref{Erohn}, but the matrix $I_{z',z'}(A',B')$ is a submatrix of $I_{z,\tilde{z}}(A,B)$ for some
	$z\in \{\pm 1\}^m,~\tilde{z}\in \{\pm 1\}^n$, which is totally non-positive of order $k$ by the hypothesis. The remainder of the
	proof is unchanged, except for the fact that we use Theorem~\ref{tnp-snr} in place of
	Theorem~\ref{tn-sign-rev_k} and the matrix $M$ is non-positive (since $I_{z,z'}(A,B)$ are non-positive).
\end{proof}

For the sake of completeness, we conclude by showing a sufficient condition for totally non-positive matrices using the Linear Complementarity Problem.
\begin{prop}\label{tnplcp1}
	Let $m,n \geq k \geq 2$ be integers and let $A \in \mathbb{R}^{m
	\times n}$ with $A\leq 0$. Then $A$ is totally non-positive of
	order $k$ if for every $r \times r$ submatrix $A_r$ of $A$ and
	for all $q \in \mathbb{R}^n$ with $q\geq 0$, if
	$z^1=(x_1,0,x_3,\ldots)^T$ and $z^2=(0,x_2,0,x_4,\ldots)^T$ (with
	all $x_i>0$) are two solutions of $\lcp{A_r}$, then $A_rz^1=A_rz^2$.
\end{prop}
\begin{proof}
	
	 By Theorem \ref{tnp-snr}, it is enough to show that every $r
	 \times r$ submatrix $A_r$ of $A$ has the non-strict sign
	 non-reversal property with respect to $\altr{r}$, for $r \in
	 [2,k]$. Fix $r \in [2,k]$. Suppose that there exist an $r\times
	 r$ submatrix $A_r$ of $A$ and a vector $x \in \altr{r}$ such
	 that $x_i(A_rx)_i<0$ for all $i \in [1,r]$. Defining
	 $x^+,x^-,v^+,v^-$ and $q$ as in the proof of Theorem
	 \ref{tn-lcp_1}, we conclude that $q\geq 0$ and $x^+,x^-$ are two
	 solutions of $\lcp{A_r}$. Also, $x^+,~x^-$ are with sign patterns $\begin{pmatrix}
	+ \\ 0 \\+\\0\\\vdots
	\end{pmatrix}$ and $\begin{pmatrix}
	0 \\ + \\0\\+\\\vdots
	\end{pmatrix}$ respectively (or vice-versa). By \eqref{thrmaeq1}, we have
	\begin{center}
		$A_rx^+ + q=v^+\neq v^-= A_rx^-+q,$
	\end{center}
	since $Ax \in \altr{r}$. Thus $A_rx^+\neq A_rx^-$, a contradiction. Hence $A$ is totally non-positive of order $k$. \end{proof}

The next result provides an improvement, via Linear Complementarity at a single vector $q$ (for each square submatrix of $A$).
\begin{prop}\label{tnplcp2}
	Let $m,n \geq k\geq 1$ be integers and $A \in \mathbb{R}^{m \times n}$. Then $A$ is totally non-positive of order $k$ if for every square submatrix $A_r$ of $A$ of size $r \in [1,k]$ and $q^{A_r}$ as defined in \eqref{tlcp}, the following holds:\begin{enumerate}
		\item[(i)] $\bf{0}$ is the solution of
		$\lcpn{A_r,q^{A_r}}$.
		\item If $z^1$ and $z^2$ are two nonzero solutions of
		$\lcpn{A_r,q^{A_r}}$ then $A_rz^1=A_rz^2$.
	\end{enumerate}
\end{prop}
\begin{proof}
	We prove by induction on $r \in [1,k]$ that $\det {A_r}\leq 0$
	for all $r \times r$ submatrices $A_r$ of $A$. Let $r=1$ and
	$A_1=(a_{ij})$ for some $i\in [1,m],~j\in [1,n]$. Then
	$x^{A_1}=(-1)$ by convention, and $q^{A_1}=-(a_{ij})$. If
	$a_{ij}> 0$ then $\bf{0}$ is not a solution of
	$\lcpn{A_1,q^{A_1}}$. Thus $A\leq 0$.
	
	Let $r \in [2,k]$ and suppose that all the $(r-1)\times (r-1)$
	and smaller minors of $A$ are non-positive, and let $A_r$ be a
	submatrix of $A$ of size $r$. If $\det {A_r}=0$, then we are
	done, else assume $A_r$ is non-singular. If $\det {A_r}>0$,
	repeating the proof of Theorem \ref{tn-lcp_2}, once again we have
	$-x^{A_r}$ and $-z^{A_r}$ (as defined in \eqref{lcpsingleeq1})
	are two distinct nonzero solutions of $\lcpn{A_r,q^{A_r}}$, but $Ax^{A_r}\neq Az^{A_r}$. Thus $\det {A_r}<0$ and hence $A$ is totally non-positive of order $k$.
\end{proof}

\begin{example}
	The converse of the previous theorem does not hold. We explain this with an example. Consider the totally non-positive matrix
	$A=\begin{pmatrix}
	0 & 0 & 0\\-1 & -3 &-3\\ -1 & -1 & -1
	\end{pmatrix}$. Then $q^A=\begin{pmatrix}
	0\\2\\2
	\end{pmatrix}$, and $z^1=\begin{pmatrix}
	2\\0\\0
	\end{pmatrix}$, $z^2=\begin{pmatrix}
	t\\0\\0
	\end{pmatrix}$, where $t>2$, are two solutions of $\lcp{A}$, but $Az^1\neq Az^2$.
\end{example}

\section{Other orthants cannot yield test vectors}

In the preceding sections, we have characterized a negative/non-positive
matrix to be totally negative/totally non-positive of order $k$, by its
square submatrices $A_{r}$ with $r \in [2,k]$ satisfying the sign
non-reversal or variation diminishing property on all of the open
bi-orthant $\altr{r}$ -- or equivalently, on single test vectors which
are drawn from this bi-orthant. Our next result explains that these
results are `best possible' in the following sense: if $x \in
\pn{r}\setminus \altr{r}$ with all $x_i\neq 0$ (i.e., there are two
successive $x_i$ of the same sign), then \textit{every} $r \times r$
totally negative matrix of order $r-1$ satisfies the sign non-reversal
property and the variation diminishing property for $x$. Thus the above
characterizations cannot hold with test vectors in any open bi-orthant other
than in $\altr{r}$,\footnote{Note that if the vector $x \neq 0$ is drawn
from an orthant that is disjoint from $\mathbb{R}^r_{+,-}$, then the
coordinates of $x$ are all non-negative or all non-positive. In both of
these cases, since $A<0$, every nonzero coordinate of $Ax$ has the
opposite sign to that of $x$. Hence $A$ does not satisfy the sign
non-reversal property with respect to such vectors $x \not\in
\mathbb{R}^r_{+,-}$.} since one shows that there exist totally negative
matrices of order $r-1$ with positive determinants (see
e.g.~\cite[Chapter V, Theorem~17]{GK50}). 

\begin{theorem}\label{thrmlast}
	 Let $A_r \in \mathbb{R}^{r \times r}$ be a totally negative matrix of order $r-1$ and let $x \in \pn{r}\setminus \altr{r}$ with all $x_i$ nonzero. Then:
	\begin{enumerate}
		\item There exists a coordinate $i \in [1,r]$ such that $x_i(A_rx)_i>0$. In other words, $A$ satisfies the sign non-reversal property.
		
		\item  $S^+(A_rx) \leq S^-(x)$. Moreover, if equality occurs, then the first (last) component of $Ax$ (if zero, the unique sign required to determine $S^+(Ax)$) has the same sign as the first (last) component of $x$. In other words, $A$ satisfies the variation diminishing property.
	\end{enumerate}
\end{theorem}
\begin{proof}
Suppose	$x \in \pn{r}\setminus \altr{r}$ with all $x_i\neq 0$. One can partition $x$ as in \eqref{vdpart} and the subsequent lines, with $k$ replaced by some $p \leq r-2$. By the argument following \eqref{vdpart}, the matrix $B \in \mathbb{R}^{r \times {(p+1)}}$ is totally negative and $Ax=B\bd{p+1}$. We define $y:=Ax=B\bd{p+1}$.

 To show $(1)$, suppose for contradiction that $x_iy_i\leq 0$ for all $i \in [1,r]$. Now consider the index set \begin{center} $I=\{i_1,\ldots, i_{p+1}\}, \hbox{ where } i_t \in [s_{t-1}+1,s_t]$.
\end{center}
Then the square submatrix $B_{I\times [p+1]}$ is totally negative and it reverses the signs of $\bd{p+1}$, since $B_{I\times [k+1]}\bd{k+1}=y_I$.  This gives the desired contradiction by Theorem \ref{tn-sign-rev_k}. 

The proof of $(2)$ is similar to that of case $(2)$ of Theorem \ref{TN-vandimnew}.
\end{proof}

An analogue of Theorem \ref{thrmlast} also holds for the class of totally non-positive matrices:
\begin{prop}\label{tnpotherorthant}
	 If $A\in \mathbb{R}^{r \times r}$ is a totally non-positive matrix of order $r-1$ and the vector $x$ is defined as in Theorem \ref{thrmlast}, then:
	\begin{enumerate}
		\item There exists a coordinate $i \in [1,r]$ such that  $x_i(A_rx)_i\geq0$. 
		\item $S^-(A_rx) \leq S^-(x)$. Moreover, if equality occurs and $Ax \neq 0$, then the first (last)
		nonzero component of $Ax$ has the same sign as the first (last) component of $x$.  
		
	\end{enumerate}
\end{prop}
\begin{proof}
	Repeat the proofs of  Theorems \ref{tnp-snr} and \ref{TNP-vandimnew}, except for the use of Theorem \ref{thrmlast} in place of Theorems \ref{tn-sign-rev_k} and \ref{TN-vandimnew} respectively.
\end{proof}

In 1950, Gantmacher--Krein \cite[Chapter V, Theorem~17]{GK50} showed that
there exist $n \times n$ totally negative matrices of order $n-1$ with
positive determinants. Thus by Theorem \ref{thrmlast} and Proposition
\ref{tnpotherorthant}, we have the following theorem.

\begin{theorem}
	Totally negative/non-positive matrices can not be characterized by the variation diminishing property or the sign non-reversal property using test vectors drawn from any open orthant outside of the bi-orthant $\altr{r}$.
\end{theorem}

We conclude with a similar observation about the LCP: Theorem
\ref{tn-lcp_1} shows that for a totally negative matrix $A$ of order $k$,
the solution sets $\sol{A_r}$ with $r\in [2,k]$ do not contain two
vectors with alternately positive and zero coordinates (and disjoint
supports) simultaneously, for all $q\in \mathbb{R}^r$ with $q>0$. Our
final result shows that if $A \in \mathbb{R}^{r \times r}$ is totally
negative of order $r-1$, the same holds when `alternating' is replaced by
`not always alternating'. Hence by the existence result alluded to in the
first paragraph of this section, totally negative matrices can not be
detected by the solution sets $\sol{A}$ of LCP, which do not contain two vectors with disjoint supports, at least one of which has two consecutive positive coordinates. Thus, 
the LCP-characterization of total negativity also distinguishes
the above `alternation'.

\begin{prop}
	If $A \in \mathbb{R}^{r \times r}$ with $r\geq 2$ is a totally
	negative matrix of order $r-1$, then $\sol{A_r}$ with $q>0$ does not simultaneously contain two vectors which have disjoint supports, and at least one of which has two consecutive positive coordinates.
\end{prop}

The proof is analogous to Theorem \ref{tn-lcp_1} using Theorem \ref{thrmlast} $(1)$.

\section*{Acknowledgments}
I thank J\"{u}rgen Garloff for pointing me to the work of
Gantmacher--Krein for existence and density results involving totally
negative matrices. I also thank Apoorva Khare for a detailed reading of
an earlier draft and for providing valuable feedback. Finally, I thank
the anonymous referee for providing useful comments and references that
improved the manuscript. This work is supported by National Post-Doctoral
Fellowship PDF/2019/000275 (SERB, Govt.\ of India), C.V.\ Raman
Postdoctoral Fellowship 80008664 (IISc), INSPIRE Faculty Fellowship research grant DST/INSPIRE/04/2021/002620
(DST, Govt.~of India),
and IIT Gandhinagar Internal Project: IP/IITGN/MATH/PNC/2223/25.

\end{document}